\documentclass[10pt,reqno]{amsart}
\usepackage{amsmath}
\usepackage{amssymb, amsthm}

\newcommand{\NN}{{\mathbb{N}}}
\newcommand{\ZZ}{{\mathbb{Z}}}

\newcommand{\RR}{{\mathbb{R}}}
\newcommand{\CC}{{\mathbb{C}}}
\newcommand{\PP}{{\mathbb{P}}}

\newcommand{\norm}[1]{{\left\| {#1} \right\|}}
\renewcommand{\d}{{\mathrm{d}}}
\newcommand{\dc}{{\mathrm{d}^c}}
\newcommand{\xb}{\overline{x}}               
\newcommand{\yb}{\overline{y}}               
\newcommand{\zb}{\overline{z}}               
\newcommand{\partialb}{\overline{\partial}}  
\newcommand{\id}{{\mathop{\mathrm{id}}}}

\newcommand{\nchar}[2]{{T_0\left( {#1}, {#2} \right)}}  
\newcommand{\ncharhol}[2]{{m\left( {#1}, {#2} \right)}} 
\newcommand{\nchardiv}[2]{{N\left( {#1}, {#2} \right)}} 
\newcommand{\ncharah}[2]{{\widetilde{m}\left( {#1}, {#2} \right)}}  
\newcommand{\ncharc}[2]{{\widetilde{m}_{\CC^2}\left( {#1}, {#2} \right)}}  

\newtheorem{theorem}{Theorem}[section]
\newtheorem*{theorem*}{Theorem}
\newtheorem{lemma}[theorem]{Lemma}
\newtheorem{proposition}[theorem]{Proposition}
\newtheorem{corollary}[theorem]{Corollary}
\theoremstyle{definition}
\newtheorem{definition}[theorem]{Definition}
\newtheorem{example}[theorem]{Example}

\newtheorem{remark}[theorem]{Remark}

\title[Holomorphic Automorphisms of Danielewski surfaces II]{Holomorphic Automorphisms of Danielewski surfaces II -- structure of the overshear group}

\author{Rafael B. Andrist}
\author{Frank Kutzschebauch}
\author{Andreas Lind}

\begin{document}
\begin{abstract}
We apply Nevanlinna theory for algebraic varieties to Danielewski surfaces and investigate their group of holomorphic automorphisms. Our main result states that the overshear group which is known to be dense in the identity component of the holomorphic automorphism group, is a free amalgamated product.
\end{abstract}\maketitle
\tableofcontents

\section{Introduction}
In the present paper we continue the study of the holomorphic automorphism group of Danielewski surfaces started in \cite{Lind}.\begin{definition}
Given a non-constant polynomial $p(z) \in \CC[z]$ with only simple zeros,
\[
D_p := \left\{ (x, y, z) \in \CC^3 \,:\, x \cdot y = p(z) \right\} \subset \CC^3
\]
with the complex structure induced from $\CC^3$, is called a \emph{Danielewski surface}.
\end{definition}
Since $p'$ and $p$ have disjoint zero sets, the surface is actually a $2$-dimensional (Stein) manifold. Danielewski surfaces have been very intensively studied
in Affine Algebraic Geometry and are in fact affine modifications of $\CC^2$.
If $p$ is affine linear, then $D_p$ is isomorphic to $\CC^2$.
If $p$ is quadratic, then $D_p$ is isomorphic to $D_q$ with $q(z) = z \cdot (z - 1)$ by an affine linear change of coordinates in $\CC^3$.

In the following we will only consider Danielewksi surfaces $D_p$ with $\deg p \geq 3$.
In \cite{Lind} we defined the notion of shears and overshears
on Danielewski surfaces and proved that the subgroup generated by overshear maps shares the following
similarity with the overshear group of $\CC^n$: it is  dense (in compact-open topology) in the identity component of the holomorphic automorphism group. For $\CC^n$ this is called the main theorem of Anders\'en-Lempert theory and is due to Anders\'en and Lempert \cite{AndersenLempert}.

\begin{definition}
A map $O_{f,g} : D_p \to D_p$ of the form
\[
O_{f,g}(x, y, z)  = \left( x, y + \frac{1}{x}\left(p(z e^{x f(x)} + x g(x)) - p(z)\right), z e^{x f(x)} + x g(x) \right)
\]
(or with the role of 1\textsuperscript{st} and 2\textsuperscript{nd} coordinates exchanged, $IO_{f,g}I$) is called an an \emph{overshear map}, where $f, g : \CC \to \CC$ are holomorphic functions (and the involution $I $ of $ D_p$  is the map interchanging $x$ and $y$) .
\end{definition}
These maps are obviously well defined maps from $D_p \to D_p$, as they can actually be continued to a holomorphic  map $\CC^3 \to \CC^3$ which can be seen by Taylor-developing the polynomial $p$ at $z$ to first order in $x$. Because of the relation $O_{f,g} \circ O_{h,k} = O_{f+h,g\cdot e^{x \cdot h}+k}$ the maps are easily seen to be holomorphic automorphisms of $D_p$. The group generated by overshears is called the overshear group $\text{OS}(D_p)$.

In the present paper we show the second similarity between the overshear group on Danielewski surfaces and
the overshear group on $\CC^2$:

\begin{theorem*}[\ref{StructureTheorem}]
Let $D_p$ be a Danielewski surface and assume that $\deg(p) \geq 4$, then the overshear group, $\text{OS}(D_p)$, is a free amalgamated product of $O_1$ and $O_2$, where $O_1$ is generated by $O_{f,g}$ and $O_2$ generated by $IO_{f,g}I$.
\end{theorem*}

In the case of $\CC^2$ a similar structure theorem is due to  Ahern and Rudin \cite{AhernRudin}. They applied this structure result to prove a global linearization result for periodic automorphisms of $\CC^2$ contained in the overshear group. For a generalization of their linearization result see \cite{KraftKu}, based on the equivariant Oka principle of Heinzner and the second author \cite{HK}.
For the final (negative) solution of the holomorphic linearization problem we refer the reader to \cite{DK1} and \cite{DK2}.

The structure result of  Ahern and Rudin generalizes the well-known algebraic result of Jung \cite{Jung} and van der Kulk \cite{VanDerKulk}. Our structure theorem generalizes an algebraic  result of Makar-Limanov \cite{MakarLimanov}. 
Following the  ideas  of Ahern and Rudin we apply Nevanlinna theory in order to prove our theorem. 

Nevanlinna theory is based on the fact that given a meromorphic function $f$, the expression $\log(|f|)$ is harmonic except in the zeroes or poles of $f$, and Jensen's formula gives the corrections needed for the mean value property of $\log(|f|)$ when zeroes or poles of $f$ are involved. Nevanlinna theory investigates the behaviour for these terms when taking the mean value over spheres with increasing radius, revealing interesting properties of the value distribution of meromorphic functions. The book of Hayman \cite{Hayman} treats the theory originally developed by Nevanlinna \cite{Nevanlinna} in one complex variable in detail, and the paper of Griffiths and King \cite{GriffithsKing} gives the generalization to Algebraic Varieties. This generalization to Algebraic surfaces is the theory we apply
in the present paper. We recall the main definitions and facts in section \ref{GriffithsKing}. Next we specialize the theory to Danielewski surfaces in sections \ref{NevanDaniel} and  \ref{NevanDaniel2}. Our main result concerning Nevanlinna theory -- which might be of independend interest -- is an estimate bounding the Nevanlinna characteristic function $\ncharah{\cdot}{\cdot}$ of a derivative of a holomorphic function by the Nevanlinna characteristic function  of the function itself:

\begin{theorem*}[\ref{MainEstimate}]
Given a holomorphic function $f : D_p \to \CC$, and a vector field $\theta$ which is a lift of a partial derivative on $\CC^2$, then we have the estimate
\begin{equation*}
\ncharah{\theta(f)}{r} \leq 14 \cdot \ncharah{f}{r} + K(n) \cdot \log(r) + L
\end{equation*}
for big $r$ outside a set of finite linear measure, where $K(n)$ is an affine polynomial and $L$ is a constant.
\end{theorem*}

In the last chapter we prove the structure theorem. As an easy application we show that the overshear group on Danielewski surfaces is a proper subgroup of the connected component of the holomorphic automorphism group by providing a concrete example of a holomorphic automorphism which is not a composition of overshears, see Corollary \ref{proper}.

\section{Nevanlinna theory for Algebraic Varieties}\label{GriffithsKing}
In this section we state the results of Griffiths and King \cite{GriffithsKing} we need.

We use the following notation:
\begin{align*}
\d  &= \phantom{\frac{i}{4 \pi} \big(} \, \partial + \partialb \\
\dc &= \frac{i}{4 \pi} \left( \partialb - \partial \right)
\end{align*}

\begin{definition}
Let $M$ be an algebraic complex manifold. 
A function $\tau : M \to [-\infty, +\infty)$ is called an \emph{exhaustion function}, if
\begin{enumerate}
\item $\tau$ is $\mathcal{C}^\infty$-smooth except for finitely many logarithmic singularities
\item the half spaces $M[r] := \left\{z \in M \,:\, e^{\tau(z)} \leq r \right\}$ are compact for $r \in [0, +\infty)$
\end{enumerate}
A logarithmic singularity of $\tau$ means that in suitable local coordinates one can write \\
 $\tau(z) = \log\norm{z} + \tilde{\tau}(z)$, where $\tilde{\tau}$ is $\mathcal{C}^\infty$-smooth.
\end{definition}

\begin{definition}
\label{def:specialexhaustion}
Let $M$ be an algebraic complex manifold of dimension $m$. 
An exhaustion function $\tau : M \to [-\infty, +\infty)$ is called a \emph{special exhaustion function}, if
\begin{enumerate}
\item\label{def:specialexhaustion1} $\tau$ has only finitely many critical points
\item\label{def:specialexhaustion2} $\d\dc\tau \geq 0$, i.e. $\tau$ is plurisubharmonic
\item\label{def:specialexhaustion3} $(\d\dc\tau)^{m-1} \not\equiv 0$ on each of the holomorphic tangent spaces to $\partial M[r]$ for all $r \in [0, \infty)$
\item\label{def:specialexhaustion4} $(\d\dc\tau)^m = 0$
\end{enumerate}
\end{definition}

\begin{definition}
Let $M$ be an algebraic complex manifold of dimension $m$ with a special exhaustion function $\tau : M \to [-\infty, +\infty)$,
and $\alpha : M \to \CC\PP^1$ a meromorphic function. We introduce the following notations:
\begin{align}
\log^+(x)     & := \max\{\log(x), 0\}, \; x \geq 0 \\
\psi          & := \d\dc\tau, \; \mbox{the Levi form of } \tau\\
\eta          & := \dc\tau \wedge \psi^{m-1}, \; \mbox{the volume form on } \partial M[r]\\
m(\alpha, r)  & := \int_{\partial M[r]} \log^+\left( \frac{1}{|\alpha|^2} \right) \eta, \; r \geq 0
\end{align}

For a divisor $D$ on $M$ we define $D[r] := D \cap M[r]$ and introduce the following notions, assuming $D$ does not pass trough any of the logarithmic singularities of $\tau$:
\begin{align}
n(D, t)         & := \int_{D[t]} \psi^{m-1} \\
\nchardiv{D}{r} & := \int_0^r \frac{n(D, t)}{t} \d t, \; \mbox{the counting function}
\end{align}
If $D$ passes trough a logarithmic singularity of $\tau$, these definitions need to be refined using Lelong numbers, as discussed in \S$1d$ of \cite{GriffithsKing}.
\end{definition}
However we shall not need this refinement using Lelong numbers in our applications to Danielewski surfaces, since the group of holomorphic automorphisms acts transitively on them \cite{Lind} and we can therefore always assume that $D$ does not pass trough a logarithmic singularity.
For meromorphic $\alpha : M \to \CC\PP^1$ and $a \in \CC\PP^1$ we denote by $D_a$ the divisor $\{\alpha(z) = a\}$.

Now one can formulate Jensen's theorem in these terms:
\begin{proposition}
With the previously defined notations, the following equation holds:
\begin{equation}
\label{eq:Jensen}
\nchardiv{D_0}{r} + \ncharhol{\alpha}{r} = \nchardiv{D_\infty}{r} + \ncharhol{1/\alpha}{r} + O(1)
\end{equation}
Let
\[
\nchar{\alpha}{r} := \nchardiv{D_\infty}{r} + \ncharhol{1/\alpha}{r},
\]
$T_0$ is called the \emph{Nevanlinna characteristic function} of $\alpha$. \\
It satisfies the following properties:
\begin{enumerate}
\item\label{eq:Nprop1} $\nchar{\alpha_1 \alpha_2}{r} \leq \nchar{\alpha_1}{r} + \nchar{\alpha_2}{r}$
\item\label{eq:Nprop2} $\nchar{\alpha_1 + \alpha_2}{r} \leq \nchar{\alpha_1}{r} + \nchar{\alpha_2}{r} + O(1)$
\item\label{eq:Nprop3} $\nchar{\alpha_1 - a}{r} = \nchar{\alpha_1}{r} + O(1), \; a \in \CC$
\item\label{eq:Nprop4} $\nchar{1/\alpha_1}{r} = \nchar{\alpha_1}{r} + O(1)$
\end{enumerate}
where $O(1)$ refers to a bounded term with respect to $r$.
\end{proposition}
To simplify notation, we define, with the divisor $D_a$ as given before:
\begin{eqnarray*}
\nchardiv{\alpha}{r} := \nchardiv{D_0}{r}
\nchardiv{1/\alpha}{r} := \nchardiv{D_\infty}{r}
\end{eqnarray*}

\begin{example}
\label{ex:exhaustc2}
For $M = \CC^n$, a special exhaustion function is given by $\tau_0 : \CC^n \to [-\infty,+\infty), \; z \mapsto \log(|z|)$. A special exhaustion function for a Danielewski surface will be given later and reduced to this special case -- it is therefore worth calculating explicitly the involved quantities for $\CC^2$: It is obvious that $\tau_0$ is an exhaustion function with a logarithmic singularity in $0$ only.
In the following we write $z = (x, y) \in \CC^2$.
\begin{enumerate}
\item Outside the logarithmic singularity there are no critical points of $\tau_0$:
\begin{eqnarray*}
\partialb \tau_0 & = & \frac{1}{2} \frac{x\d\xb + y\d\yb}{x\xb + y\yb} \\
\partial  \tau_0 & = & \frac{1}{2} \frac{\xb\d x + \yb\d y}{x\xb + y\yb} \\
\end{eqnarray*}
\item
\[
\d \dc \tau_0 = 
\frac{i}{2 \pi} \partial \partialb \tau_0 =  \frac{i}{2 \pi} \frac{y\yb \d x \wedge \d \xb + x \xb \d y \wedge \d \yb - x \yb \d y \wedge \d \xb - y \xb \d x \wedge \d \yb}{2 (x\xb + y\yb)^2} \geq 0
\]
\item $\partial M[r] = \left\{ x\xb + y\yb = r^2 \right\}$ and the holomorphic tangent space in $(x, y)$ is given by the relation $\xb \d x + \yb \d y = 0$, therefore:
\[
\d \dc \tau_0 = \frac{i}{2 \pi r^2} \d y \wedge \d\yb \not \equiv 0 \mbox{ on the holomorphic tangent space}
\].
\item $(\d \dc\tau_0)^2  = 0$ can be seen by explicit calculation.
\end{enumerate}
For the volume form $\eta_0$ we get:
\begin{eqnarray*}
\eta_0 := \dc \tau_0 \wedge \d \dc \tau_0 & = & \frac{1}{8 \pi^2} \partial \partialb \tau_0 \wedge (\partial - \partialb) \tau_0 \\
& = & \frac{1}{(4\pi)^2 (x \xb + y \yb)^2} \Big( \yb \d x \wedge \d \xb \wedge \d y - y \d x \wedge \d \xb \wedge \d \yb \\ &   &  \qquad \quad \quad \; \; \; \, + \; \xb \d x \wedge \d y \wedge \d \yb - x \d \xb \wedge \d y \wedge \d \yb  \Big)
\end{eqnarray*}
It is rotation invariant since $\tau_0$ is, and scales such that
$\eta_0(rx, ry) = \eta_0(x, y)$, and in particular we have that
\[
\int_{r S^3} \eta_0 = r^3 \cdot \int_{S^3} \eta_0 = 2 r^3
\]
\end{example}

\section{Nevanlinna theory for Danielewski Surfaces}\label{NevanDaniel}
For a Danielewski surface $D_p$ we define a projection $\pi : D_p \to \CC^2$ as
\[
\pi(x, y, z) = (x + y, z)
\]
and for the special exhaustion function we choose
\begin{align*}
\tau : D_p & \to [-\infty, +\infty) \\
   (x,y,z) & \mapsto \log\left(|\pi(x, y, z)|\right)
\end{align*}

\begin{lemma}
This function $\tau$ defined above is a special exhaustion function in the sense of Definition \ref{def:specialexhaustion}.
\end{lemma}
\begin{proof}
$\log\left(|\pi(x, y, z)|\right)$ can only have logarithmic singularities when $\pi(x, y, z) = (0, 0)$, i.e.
$x + y = 0, z = 0$ and therefore $x \cdot y = p(0)$ which happens only for two points. Let $\tau_0 : \CC^2 \to [-\infty, +\infty)$ be the special exhaustion function defined in Example \ref{ex:exhaustc2}. We can then write $\tau = \tau_0 \circ \pi$. The projection $\pi$ is linear, therefore the results of calculations are the same as in the case of $\CC^2$, when in the end $(x, y)$ is formally replaced by $(x + y, z)$. Now we check the properties required by Definition \ref{def:specialexhaustion}:
\begin{enumerate}
\item Critical points can occur only in case of $x + y = 0$ and $z = 0$:
\begin{eqnarray*}
\partialb \tau & = & \frac{1}{2} \frac{(x+y)(\d\xb + \d\yb) + z\d\zb}{(x + y)(\xb + \yb) + z\zb} \\
\partial  \tau & = & \frac{1}{2} \frac{(\xb+\yb)(\d x+\d y)+ \zb\d z}{(x + y)(\xb + \yb) + z\zb} \\
\end{eqnarray*}
So there are in fact no critical points at all (outside the two logarithmic singularities).
\item $\d \dc \tau  \geq 0$ follows directly from the linearity of $\pi$.
\item $\partial M[r] = \left\{ (x + y)(\xb + \yb) + z\zb = r^2, x \cdot y = p(z) \right\}$ and the holomorphic tangent space in $(x, y, z)$ is given by the relations $(\xb + \yb) (\d x + \d y) + \zb \d z = 0$ and $y \d x + x \d y = p'(z) \d z$, hence:
\[
\d \dc \tau = \frac{i}{2 \pi r^2} \d z \wedge \d\zb \not \equiv 0 \mbox{ on the holomorphic tangent space}
\].
\item $(\d \dc \tau)^2 = 0$ follows directly from the linearity of $\pi$. \qedhere
\end{enumerate}
\end{proof}
Now we want to write out more explicitly the Nevanlinna characteristic function for a holomorphic function $f : D_p \to \CC$ where the expression reduces to
\begin{eqnarray*}
\nchar{f}{r} & = & \ncharhol{1/f}{r} = \int_{\partial D_p[r]} \log^+\left( |f|^2 \right) \eta \\
 & = & \int_{\partial D_p[r]} \log^+\left( |f|^2 \right) \pi^*(\eta_0) \\
 & = & \int_{\pi^{-1}(r S^3)} \log^+{|f|^2} \pi^*(\eta_0) \\
\end{eqnarray*}
In order to apply this Nevanlinna characteristic function to overshears in the next section, we need to estimate derivatives with respect $r$. This can be calculated easier in an even more explicit form. First we need to investigate the projection $\pi$ a bit more.

Ramification points of $\pi$: Given $(a, b) \in \CC^2$, we look for the pre-images $(x, y, z) \in D_p$, i.e. $(x, y, z) \in \CC^3$ such that
\begin{align*}
a    &= x + y \\
b    &= z \\
p(z) &= x \cdot y
\end{align*}
This obviously leads to a quadratic equation, symmetric in $x$ and $y$. Therefore ramification points occur exactly if $x=y$, otherwise the covering has 2 sheets. For $(a,b) \in \CC^2$ this means that $p(b) = \frac{a^2}{4}$. We cut out these points in the following way:
For every $(a,b) \in \CC^2$ satisfying $p(b) = \frac{a^2}{4}$, we cut along the real $a$-line towards $-\infty$. The set of points we cut out is denoted by $C$, a set of real dimension $3$. Now, $\pi : \pi^{-1}(\CC^2 \setminus C) \to \CC^2 \setminus C$ is an unbranched 2-sheeted covering, and $\CC^2 \setminus C$ is simply connected. Since a line is always either tangential or transversal to a sphere, the intersection of $C$ with $rS^3 \subset \CC^2$ is always a set of real dimension $2$ and therefore of measure zero with respect to $\eta_0$, hence its removal does not affect the integral. By $s_j : \CC^2 \setminus C \to D_p \; (j = 1,2)$ we now denote the sections corresponding to this restricted unbranched covering s.t. $\pi \circ s_j = \mathrm{id}$.
\begin{eqnarray*}
\ncharhol{1/f}{r} & = & \int_{\pi^{-1}(r S^3)} \log^+{|f|^2} \pi^*(\eta_0) \\
 & = & \int_{r S^3 \setminus C} \sum_{j=1}^2 \log^+{|f \circ s_j|^2} \eta_0
\end{eqnarray*}

\begin{definition}
Ahern and Rudin \cite{AhernRudin} used a slightly different definition of the Nevanlinna characteristic function for holomorphic $f : \CC^n \to \CC$, namely
\[
\ncharah{f}{r} := \int_{S^{2n-1} \subseteq \CC^n} \log^+|f(r \zeta)| \d \sigma(\zeta)
\]
where $\sigma$ is a finite rotational invariant Borel measure on the sphere $S^{2n-1}$.
\end{definition}
By choosing an appropriate normalisation for the measure $\sigma$ we can arrange that on $\CC^n$ holds the following relation:
\begin{equation}
 r^{2n-1} \cdot \ncharah{f}{r} = \ncharhol{1/f}{r}
\end{equation}
We choose this as a definition for $\ncharah{f}{r}$ on a Danielewski surface ($n=2$). All properties so far derived for $\ncharhol{f}{r}$ are inherited by $\ncharah{f}{r}$, except the growth rate in $r$ which gets modified by $r^3$.

\begin{lemma}
\label{lem:holNevanlinnaproperties}
Let $f, g : D_p \to \CC$ be holomorphic functions, then their Nevanlinna characteristic satisfies the following properties:
\begin{enumerate}
\item \label{Nevanlinnaproperty1} $\ncharah{f}{r} - \ncharah{g}{r} + O(1) \leq \ncharah{f + g}{r} \leq \ncharah{f}{r} + \ncharah{g}{r} + O(1)$
\item \label{Nevanlinnaproperty2} $\ncharah{f \cdot g}{r} \leq \ncharah{f}{r} + \ncharah{g}{r} + O(1)$
\item \label{Nevanlinnaproperty3} $\ncharah{1/f}{r} \leq \ncharah{f}{r} + O(1)$
\end{enumerate}
\end{lemma}
\begin{proof}
The first two properties follow directly from
\[
\log(x) = \log^+(x) - \log^+(1/x), \; x \in \RR^*_+
\]
and the thereby from inherited inequalities of the logarithm:
\[
\log^+|z + w| \leq \log^+(2 \max\{|z|,|w|\}) \leq \log^+|z| + \log|w| + \log 2, \; z, w, z+w \in \CC^*
\]
\[
\log^+|z w| \leq \log^+|z| + \log^+|w|, \; z, w \in \CC^*
\]
For property \ref{Nevanlinnaproperty3} we need Jensens's formula \eqref{eq:Jensen}:
\begin{multline*}
r^3 \ncharah{1/f}{r} = \ncharhol{f}{r} = \nchar{1/f}{r} \\
 = \nchar{f}{r} + O(1) = \ncharhol{1/f}{r} + \nchardiv{1/f}{r} + O(1) 
\end{multline*}
Now observe that $\frac{1}{r^3}\nchardiv{1/f}{r}$ is of order $O(1)$, since the counting function of a such divisor associated $1/f$ has support at most  in a hypersurface. \qedhere
\end{proof}

Using these elementary properties we are now able to prove
\begin{proposition}
\label{prop:Mohonko}
Let $f : D_p \to \CC$ be a meromorphic function, and let
\[
R_d(z, f(z)) = a_0(z) + a_1(z)f(z) + \cdots + a_d(z) (f(z))^d
,\] where $a_i : D_p \to \CC$ are meromorphic, and assume that
$a_d \not\equiv 0$. Then
\[
\ncharah{R_d(z,f(z))}{r} = d \cdot \ncharah{f}{r} + \sum_{i=0}^d O(\ncharah{a_i}{r}) + O(1)
\]
\end{proposition}

In the case of meromorphic functions defined on $\CC$, this proposition was first proved by Mohon'ko \cite{Mohonko}. Mohon'ko's paper is in Russian, but an English proof can be found in \cite{Laine}. We adapt the same proof to our more general situation.

\begin{lemma}
\label{lem:squarecompl}
Let
\[
A(z,f) := (\varphi_1(z)f + \cdots + \varphi_{d-1}(z)f^{d-1} + f^d)f^{d-2}
,\]
be a polynomial in $f$ with meromorphic
coefficients. Then there exist \\
$u_0, \dots, u_{d-1}, q_0, \dots, q_{d-2}$ which are polynomials in $\varphi_1, \dots, \varphi_{d-1}$
with constant coefficients, such that
\[B(z,f) := u_0(z) + \cdots + u_{d-1}(z)f^{d-1}\]
satisfies
\[(B(z,f))^2 = A(z,f) + \sum_{i=0}^{d-2} q_i(z)f^i\]
\end{lemma}

\noindent The proof of Lemma \ref{lem:squarecompl} can be found in
either \cite{Mohonko} or \cite{Laine}, but it is not so hard to see
that $u_{d-1} \equiv 1$, $u_{d-2}(z) = \frac{1}{2}\varphi_{d-1}(z)$,
\[
 u_{d-k}(z) = \frac{1}{2} \cdot \left(\varphi_{d-k+1}(z) - u_{d-k+1}^2(z)\right), \quad k= 3, \dots, d
\]
and
\[
 q_i(z) = u_0(z)u_i(z) + u_1(z)u_{i-1}(z) + \cdots + u_i(z)u_0(z)
\]
solve the problem.

\begin{proof}[Proof of Proposition \ref{prop:Mohonko}]
\begin{eqnarray*}
\ncharah{\sum_{i=0}^d a_i f^i}{r} & \leq & \ncharah{f\cdot\sum_{i=1}^d a_i f^{i-1}}{r} + \ncharah{a_0}{r} + O(1) \\
& \leq & \ncharah{f}{r} + \ncharah{\sum_{i=1}^d a_i f^{i-1}}{r} + \ncharah{a_0}{r} + O(1)
\end{eqnarray*}
By induction
\[\ncharah{\sum_{j=0}^d a_j f^j}{r} \leq d \cdot \ncharah{f}{r} + \sum_{i=0}^d
O(\ncharah{a_i}{r}) + O(1)\]

\noindent Conversely, assume first that $d=1$. Then
\begin{multline*}
\ncharah{f}{r} = \ncharah{\frac{R_1 - a_0}{a_1}}{r} 
 \leq \ncharah{R_1}{r} + \ncharah{a_0}{r} + \ncharah{a_1}{r} + O(1).
\end{multline*}
Rearranging this inequality proves the proposition for $d=1$.

Now, assume that the proposition has been proved for all polynomials
$P(z,f)$, as in the statement, of degree $s \leq d-1$ in $f$. That
is
\[
\ncharah{P(z,f)}{r} = s \cdot \ncharah{f}{r} + \sum_{j=0}^{d-1} O(\ncharah{a_j}{r}) + O(1)
.\]
Observe that
\[
\frac{R_d - a_0}{a_d} f^{d-2} = \left(\varphi_1 f + \cdots + \varphi_{d-1}f^{d-1} + \varphi_d f^d\right) \cdot f^{d-2}
,\]
where $\varphi_j = a_j/a_d$ for $j=0,\dots, d-1$.
Using Lemma \ref{lem:squarecompl} we see that
\[
(B(z,f))^2 = \frac{R_d-a_0}{a_d}f^{d-2} + \sum_{i=0}^{d-2}
q_i(z)f^i
,\]
where the degree of $B(z,f)$ in $f$ is $d-1$. By
the induction hypothesis we get
\begin{multline*}
\ncharah{(B(z,f))^2}{r} = 2\cdot\ncharah{B(z,f)}{r} \\
 = 2(d-1)\cdot\ncharah{f}{r} + \sum_{i=0}^{d-1} O(\ncharah{a_i}{r}) + O(1).
\end{multline*}
On the other hand,
\begin{multline*}
\ncharah{(B(z,f))^2}{r}
 \leq (d-2)\ncharah{f}{r} + \ncharah{\frac{R_d-a_0}{a_d}}{r} + \sum_{i=0}^{d-2} \ncharah{q_i}{r} + O(1).
\end{multline*}
Looking at the proof of Lemma \ref{lem:squarecompl} and performing an obvious subtraction gives
\[
\ncharah{R_d}{r} \geq d\cdot\ncharah{f}{r} + \sum_{i=0}^d O(\ncharah{a_i}{r}) + O(1)
\]
Now we have shown both inequalities.
\end{proof}

\begin{corollary}
\label{cor:polycomp}
Let $q \in \CC[z]$ be a polynomial, and $f : D_p \to \CC$ a holomorphic function, then
\[
\ncharah{q \circ f}{r} = \deg(q) \cdot \ncharah{f}{r} + O(1)
\]
\end{corollary}

\begin{proposition}
\label{prop:transcendental}
Let $f : D_p \to \CC$ be a holomorphic function and $g: \CC \to \CC$ transcendental. Then
\[
\lim_{r \to \infty} \frac{\nchar{g \circ f}{r}}{\nchar{f}{r}} = \lim_{r \to \infty} \frac{\ncharah{g \circ f}{r}}{\ncharah{f}{ r}} = \infty
\]
\end{proposition}
\begin{proof} \hfill
\begin{enumerate}
\item Let $g : \CC \to \CC$ be transcendental. By Picard's theorem we can assume without loss of generality that $g$ has infinitely many zeros, otherwise we consider instead $z \mapsto g(z) + c, \, c \in \CC^*,$ and note that $\nchar{\alpha}{r} = \nchar{\alpha + c}{r} + O(1)$. Let $\left\{w_j\right\}_{j=1}^\infty$ be the zeros of $g$.

\item Claim:
\begin{equation}
\label{eq:ncharhol1}
\ncharah{\prod_{j=1}^n \frac{1}{f - w_j}}{r} \leq  \ncharah{\frac{1}{g \circ f}}{r} + O(1)
\end{equation}
Proof:
\begin{multline*}
   \ncharah{\frac{1}{g \circ f}}{r} - \ncharah{\prod_{j=1}^n \frac{1}{ f - w_j }}{r} \\
 = 2 \int_{\pi^{-1}(rS^3)} \left( \log^+\left| \frac{1}{(g \circ f)(z)} \right| - \log^+\left| \prod_{j=1}^n \frac{1}{ f(z) - w_j} \right| \right) \eta(z)
\end{multline*}
Recall that $\eta$ is a positive volume form. Hence a negative integrand would imply $\displaystyle \left|\prod_{j=1}^n \frac{1}{f(z) - w_j}\right| > 1$. The set $\Omega \subseteq \CC$ defined by \\ $\left\{w \in \CC \,:\, \left|\prod_{j=1}^n (w - w_j)\right| < 1 \right\}$ is bounded, hence there exists an $A > 0 \,:\, |g \circ f|\leq A$ on $\overline{\Omega}$. The integrand is certainly smaller if we integrate only over the set such that the integrand is negative, i.e. we continue with the estimates:
\begin{multline*}
\geq 2 \int_{f^{-1}(\Omega) \cap \pi^{-1}(r S^3)} \left( \log^+\left| \frac{1}{(g \circ f)(z)} \right| - \log^+\left| \prod_{j=1}^n \frac{1}{ f(z) - w_j} \right| \right) \eta(z) \\
\geq 2 \int_{\{z \in D_p \,:\, f(z) \in \Omega\} \cap \pi^{-1}(r S^3)} \log\left|\frac{1}{g \circ f}\right| \eta \geq 2 \log \frac{1}{A},
\end{multline*}
where $A$ depends on $n$, but not on $r$, and this proves the claim.
\item We state two facts about $N(\cdot, r)$, they are immediate consequences of the fact that this is the counting function of a divisor.
\begin{eqnarray}
\label{eq:nchardivfact1}
\nchardiv{\frac{1}{g \circ f}}{r} & \geq & \nchardiv{\prod_{j=1}^n \frac{1}{f - w_j}}{r} \\
\label{eq:nchardivfact2}
\nchardiv{\prod_{j=1}^n \frac{1}{f - w_j}}{r} & = & \sum_{j=1}^n \nchardiv{\frac{1}{f - w_j}}{r}
\end{eqnarray}
\item
By Jensen's formula \eqref{eq:Jensen} we have:
\begin{multline}
\nchardiv{\prod_{j=1}^n \frac{1}{f - w_j}}{r} + \ncharah{\prod_{j=1}^n \frac{1}{f - w_j}}{r} \\
 = \underbrace{\nchardiv{\prod_{j=1}^n (f - w_j)}{r}}_{=0} + \ncharah{\prod_{j=1}^n (f - w_j)}{r} + O(1)
\end{multline}
as well as
\begin{multline}
\nchardiv{\frac{1}{f - w_j}}{r} + \ncharah{\frac{1}{f - w_j}}{r} \\
 = \underbrace{\nchardiv{f - w_j}{r}}_{=0} + \ncharah{f - w_j}{r} + O(1)
\end{multline}
Summing up the last equation over $j$ and using Jensen's formula again, we obtain:
\begin{multline*}
\sum_{j=1}^n \nchardiv{ \frac{1}{f - w_j}}{r} + \sum_{j=1}^n \ncharah{ \frac{1}{f - w_j}}{r} \\
 = \sum_{j=1}^n \ncharah{ f - w_j}{r} + O(1)
 = \ncharah{\prod_{j=1}^n (f - w_j)}{r} + O(1) \\
 = \nchardiv{\prod_{j=1}^n \frac{1}{f - w_j}}{r} + \ncharah{\prod_{j=1}^n \frac{1}{f - w_j}}{r} + O(1)
\end{multline*}
\begin{equation}
\label{eq:ncharhol2}
\mbox{which implies } \sum_{j=1}^n \ncharah{ \frac{1}{f - w_j}}{r} = \ncharah{\prod_{j=1}^n \frac{1}{f - w_j}}{r} + O(1)
\end{equation}
The inequality \eqref{eq:ncharhol1} and equation \eqref{eq:ncharhol2} together finally give:
\begin{equation}
\label{eq:ncharholtrans}
\ncharah{\frac{1}{g \circ f}}{r} \geq \sum_{j=1}^n \ncharah{\frac{1}{f - w_j}}{r} + O(1)
\end{equation}
\item
Summing up inequalities \eqref{eq:ncharholtrans} and \eqref{eq:nchardivfact1}, we obtain:
\[\nchar{\frac{1}{g \circ f}}{r} \geq \sum_{j = 1}^n \nchar{\frac{1}{f - w_j}}{r} + O(1)\]
and by Jensen's formula, this is equivalent to
\begin{eqnarray*}
\nchar{g \circ f}{r} & \geq & \sum_{j=1}^n \nchar{f - w_j}{r} + O(1) \\
             &  =  & n \cdot \nchar{f}{r} + O(1)
\end{eqnarray*}
Therefore we have for all $n \in \NN$:
\[\lim_{r \to \infty} \frac{\nchar{g \circ f}{r}}{\nchar{f}{r}} \geq n \qedhere\] 
\end{enumerate}
\end{proof}

\begin{lemma}
\label{lem:coordinates}
Let $p$ be a polynomial of degree $n$, $D_p$ the corresponding Danielewski surface.
By $x, y, z : D_p \to \CC$ we refer to the coordinate functions of $\CC^3$, restricted to $D_p$. Then the growth of the Nevanlinna characteristic function can be described as follows:
\begin{enumerate}
 \item $\displaystyle \lim_{r \to \infty} \frac{\ncharah{x}{r}}{\ncharah{y}{r}} = 1$
 \item $\displaystyle \lim_{r \to \infty} \frac{\ncharah{x}{r}}{\ncharah{z}{r}} = \lim_{r \to \infty} \frac{\ncharah{y}{r}}{\ncharah{z}{r}} = \frac{n}{2}$
 \item $O\left( \ncharah{z}{r} \right) = O(\log r)$
\end{enumerate}

\end{lemma}
\begin{proof}
The relations between the three Nevanlinna characteristic functions follow directly from $x \cdot y = p(z)$ and Proposition \ref{prop:Mohonko}. It is therefore sufficient to calculate the asymptotic behaviour of $\ncharah{z}{r}$ which is, using the covering $\pi : D_p \to \CC^2$, twice the Nevanlinna characteristic function of a coordinate in $\CC^2$:
\[
\ncharah{f}{r} = \int_{S^{3} \subseteq \CC^2} \log^+|f(r \zeta)| \d \sigma(\zeta) \leq \sigma(S^3) \cdot \log(r), \quad r \geq 1
\]
with $f(z_1, z_2) = z_1$. 
\end{proof}

\section{Nevanlinna Characteristic Function of Derivatives} \label{NevanDaniel2}

Ahern and Rudin \cite{AhernRudin} showed
\begin{proposition}
\label{prop:derivativeAhernRudin}
Let $f : \CC^n \to \CC$ be holomorphic, then
\[
\ncharah{\frac{\partial f}{\partial z_j}}{r} \leq 3 \cdot \ncharah{f}{r} + n \log r 
\]
\end{proposition}
The objective for the rest of this section is to derive a similar estimate in the case of Danielewski surfaces.
Let $(a, b) \in \CC^2$ be coordinates on $\CC^2$. Using the the branched covering $\pi \colon D_p \to \CC^2$ to lift the vector fields $\frac{\partial}{\partial a}$ and $\frac{\partial}{\partial b}$ we get the following meromorphic vector fields on $D_p$:
\begin{equation}\label{LiftedParialDerivative1}
\theta_1 = \frac{x}{x-y} \frac{\partial}{\partial x} - \frac{y}{x-y}\frac{\partial}{\partial y}
\end{equation}
and
\begin{equation}\label{LiftedParialDerivative2}
\theta_2 = \frac{p'(z)}{y-x} \left( \frac{\partial}{\partial x} - \frac{\partial}{\partial y} \right) + \frac{\partial}{\partial z}.
\end{equation}

Since $\theta_i$ correspond to partial derivatives, we want to estimate, in the spirit of Proposition \ref{prop:derivativeAhernRudin}, $\ncharah{\theta_i(f)}{r}$ to generalize Ahern and Rudins Proposition to Danielewski surfaces.

\begin{theorem}
\label{MainEstimate}
Given a holomorphic function $f : D_p \to \CC$, and a vector field $\theta$ which is a lift of a partial derivative on $\CC^2$, then we have the estimate
\begin{equation*}
\ncharah{\theta(f)}{r} \leq 14 \cdot \ncharah{f}{r} + K(n) \cdot \log(r) + L
\end{equation*}
for big $r$ outside a set of finite linear measure, where $K(n)$ is an affine polynomial and $L$ is a constant.
\end{theorem}

\noindent The proof of Theorem~\ref{MainEstimate} contains several steps. First an observation.

\begin{remark}\label{InvariantObs}
Recall the involution map $I :D_p \to D_p$ defined by $I(x,y,z) = (y,x,z)$. Given a function $f \colon D_p \to \CC$ we can decompose this function in an $I$-invariant and an $I$-anti-invariant part by
\[f = \frac{f + f \circ I}{2} + \frac{f - f \circ I}{2} = f_\text{inv} + f_\text{anti}\;,\]
where $f_\text{inv} := \frac{f + f \circ I}{2}$ is $I$-invariant while $f_\text{anti} := \frac{f - f \circ I}{2}$ is anti-invariant under $I$. Clearly, the invariance respectively anti-invariance means that
\[f_\text{inv} \circ I = f_\text{inv}\]
respectively
\[f_\text{anti} \circ I = - f_\text{anti}.\]
For simplicity, in the following we will write \emph{invariant} and \emph{anti-invariant} instead of $I$-\emph{invariant} and $I$-\emph{anti-invariant}.
\end{remark}

\begin{lemma}\label{InvolutionNevanlinna}
The Nevanlinna characteristic of $f$ and $f \circ I$ is equal, i.e.
\[
\ncharah{f}{r} = \ncharah{f \circ I}{r}
\]
for all $f \in \mathcal{O}(D_p)$.
\end{lemma}
\begin{proof}
This is a simple consequence of the definition of $\ncharah{f}{r}$ resp. $\ncharhol{1/f}{r}$ as an integral over $D_p$.
\end{proof}

\noindent The following lemma is trivial:

\begin{lemma}\label{InvariantLemma}
Every vector field $\theta$ on $D_p$ which is a lift from a vector field on $\CC^2$ is invariant under the involution $I$, i.e. $I_*\theta = \theta$.
\end{lemma}

Given a function $f \colon D_p \to \CC$, decompose it as in Observation~\ref{InvariantObs}, so $f = f_\text{inv} + f_\text{anti}$. As every lifted vector field $\theta$ on $D_p$ is invariant by Lemma~\ref{InvariantLemma}, we get that $\theta(f_\text{inv})$ is invariant and that $\theta(f_\text{anti})$ is anti-invar\-iant. Therefore, by linearity and property \ref{Nevanlinnaproperty1} in Lemma~\ref{lem:holNevanlinnaproperties} we get
\begin{multline}\label{VectorfieldSplit}
\ncharah{\theta(f)}{r} = \ncharah{\theta(f_\text{inv}) + \theta(f_\text{anti})}{r} \\
 \leq \ncharah{\theta(f_\text{inv})}{r} + \ncharah{\theta(f_\text{anti})}{r} + C.
\end{multline}

By expression~\eqref{VectorfieldSplit} we need to estimate the invariant and anti-invariant part separately to prove Theorem~\ref{MainEstimate}. We prove these different cases in two lemmas.

\begin{lemma}\label{InvariantResult}
Let $\theta$ be a lift of a partial derivative on $\CC^2$ and let $f : D_p \to \CC$ be an invariant holomorphic function. Then
\[
\ncharah{\theta(f)}{r} \leq 3 \ncharah{f}{r} + 4 \log(r)
\]
for all $r$ outside a set of finite linear measure.
\end{lemma}
\begin{proof}
The function $f : D_p \to \CC$ defines a holomorphic function $\tilde{f} : \CC^2 \to \CC$ by
\[
\tilde{f}(a, b) := f(\pi^{-1}(a, b))
.\]
Then
\begin{equation}
\label{DanToC2Obs}
\ncharah{f}{r} = 2 \ncharc{\tilde{f}}{r},
\end{equation}
where $\ncharc{\cdot}{\cdot}$ refers to the Nevanlinna characteristic function on the Danielewski surface $\CC^2$.
Indeed, we have
\begin{multline*}
\ncharc{\tilde{f}}{r} = \int_{S^3} \log^+ |f(\pi^{-1}(r\zeta))|d \sigma \\
 = \int_{\bar x \in \pi^{-1}(rS^3)} \log^+|f(\bar x)| d(\pi^*\sigma)(\bar x) = \frac{1}{2} \ncharah{f}{r}.
\end{multline*}
By Proposition \ref{prop:derivativeAhernRudin} we know that
\[
\ncharc{\frac{\partial \tilde{f}}{\partial z_i}}{r} \leq 3 \ncharc{\tilde{f}}{r} + 2\log(r)
\]
for $r$ outside a set of finite linear measure. Assume that $\theta$ is a lift of a partial derivative, then $\theta(f)$ is invariant by the discussion after Lemma~\ref{InvariantLemma} and $\theta(f) = \pi^*\left(\frac{\partial \tilde{f}}{\partial z_i}\right)$. By expression~\eqref{DanToC2Obs} we get that
\begin{equation}
\label{EndOfProof1}
\ncharah{\theta(f)}{r} = 2 \cdot \ncharc{\frac{\partial \tilde{f}}{\partial z_i}}{r} \leq 6 \cdot \ncharc{\tilde{f}}{r} + 4\log(r)
\end{equation}
by Proposition \ref{prop:derivativeAhernRudin}, and $r$ outside a set of finite linear measure. Going back to the characteristic function on $D_p$ we get that the right hand side of inequality~\eqref{EndOfProof1} equals
\[
3 \cdot\ncharah{f}{r} + 4\log(r).
\]
Hence the lemma is proved.
\end{proof}

\begin{lemma}\label{AntiInvariantResult}
Let $\theta$ be a lift of a partial derivative on $\CC^2$ and let $f \colon D_p \to \CC$ be an anti-invariant holomorphic function. Then
\[
\ncharah{\theta(f)}{r} \leq 4 \ncharah{f}{r} + E(n) \cdot \log(r) + F
\]
for big $r$ outside a set of finite linear measure, where $E(n)$ is an affine polynomial and $F$ is a constant.
\end{lemma}
\begin{proof}
Let $f : D_p \to \CC$ be an anti-invariant holomorphic function. Then $(x-y) f$ is invariant. Let $\theta$ be a lift of a partial derivative, and therefore $\theta((x-y)f)$ is invariant by the discussion after Lemma~\ref{InvariantLemma}. By Lemma~\ref{InvariantResult} we get that
\begin{equation}
\label{ThetaEstimateAgain}
\ncharah{\theta((x-y)f)}{r} \leq 3 \cdot \ncharah{(x-y)f}{r} + 4 \log(r)
\end{equation}
for all $r$ outside a set of finite linear measure. Since $\theta$ is a vector field it fulfils Leibniz rule.
\begin{equation}
\label{LeibnizVectorfield}
\theta((x-y)f) = f \cdot \theta(x-y) + (x-y) \cdot \theta(f)
\end{equation}
We get, using property (\ref{Nevanlinnaproperty2}) in Lemma~\ref{lem:holNevanlinnaproperties}, that
\begin{multline}
\label{EqAnti1}
\ncharah{\theta(f)}{r} = \ncharah{\frac{(x-y)\theta(f)}{x-y}}{r}
 \leq \ncharah{\frac{1}{x-y}}{r} + \ncharah{(x-y)\theta(f)}{r}
\end{multline}
Now use (\ref{Nevanlinnaproperty3}) in Lemma~\ref{lem:holNevanlinnaproperties}. Then we get that the right hand side of \eqref{EqAnti1} is less or equal to
\begin{multline}
\label{EqAnti2}
\ncharah{x-y}{r} + \ncharah{(x-y)\theta(f)}{r} + C \\
 = \ncharah{x-y}{r} + \ncharah{\theta((x-y)f) - f \cdot \theta(x-y)}{r} + C
\end{multline}
where we in the last equality used expression~\eqref{LeibnizVectorfield}. Using properties (\ref{Nevanlinnaproperty1}) and (\ref{Nevanlinnaproperty2}) in Lemma~\ref{lem:holNevanlinnaproperties} we get that the right hand side of \eqref{EqAnti2} is less or equal to
\begin{equation}
\label{EqAnti3}
\ncharah{x-y}{r} + \ncharah{\theta((x-y)f)}{r} + \ncharah{f}{r}
 + \ncharah{\theta(x-y)}{r} + \tilde{C}.
\end{equation}
By equation~\eqref{ThetaEstimateAgain} we get that the right hand side of \eqref{EqAnti3} is less or equal to
\begin{multline}
\label{EqAnti4} 
\ncharah{x-y}{r} + 3 \cdot \ncharah{(x-y)f}{r} + 4 \cdot \log(r) + \ncharah{f}{r} \\
 + \ncharah{\theta(x-y)}{r} + \tilde{C}.
\end{multline}
for all $r$ outside a set of finite linear measure. Using \ref{Nevanlinnaproperty2} from Lemma~\ref{lem:holNevanlinnaproperties} we get that the right hand side of \eqref{EqAnti4} is less or equal to
\begin{equation}
\label{EqAnti5}
4\cdot\ncharah{x-y}{r} + 4 \cdot \log(r) + 4\cdot\ncharah{f}{r} + \ncharah{\theta(x-y)}{r} + \tilde{C}.
\end{equation}
By (\ref{Nevanlinnaproperty2}) in Lemma~\ref{lem:holNevanlinnaproperties} and by Lemma~\ref{lem:coordinates} we get that the right hand side of \eqref{EqAnti5} is less or equal to
\begin{multline}
\label{EqAnti6}
4 \cdot (2n \cdot \log(r) + D) + 4 \cdot \log(r) + 4\cdot\ncharah{f}{r} + \ncharah{\theta(x-y)}{r} + \tilde{C} \\
 = (8n + 4) \cdot \log(r) + 4 \cdot \ncharah{f}{r} +\ncharah{\theta(x-y)}{r} + \tilde{D}
\end{multline}

We need to estimate $\ncharah{\theta(x-y)}{r}$, so we consider two cases, namely $\theta = \theta_1$ or $\theta=\theta_2$ from equation~\eqref{LiftedParialDerivative1} and equation~\eqref{LiftedParialDerivative2}. When $\theta=\theta_1$ we get that
\begin{equation}
\label{EqAnti7}
\ncharah{\left( \frac{x}{x-y} \frac{\partial}{\partial x} - \frac{y}{x-y}\frac{\partial}{\partial y}\right) (x-y)}{r}
 = \ncharah{\frac{x+y}{x-y}}{r}
\end{equation}
By property (\ref{Nevanlinnaproperty1}) and (\ref{Nevanlinnaproperty3}) in Lemma~\ref{lem:holNevanlinnaproperties} we get that the right hand side of \eqref{EqAnti7} is less or equal to
\begin{equation}
\label{Theta1Eq}
2\cdot\ncharah{x}{r} + 2\cdot\ncharah{y}{r} + E = 4n \cdot\log(r) + E
\end{equation}
where we in the last equality used Lemma~\ref{lem:coordinates} for big $r$ outside a set of finite linear measure. Combining equations \eqref{Theta1Eq} and \eqref{EqAnti6} we get, by following the chain of inequalities from expression~\eqref{EqAnti1}, that
\begin{equation}
\label{ThetaEq1}
\ncharah{\theta(f)}{r} \leq 4\cdot\ncharah{f}{r} + (12n+4)\cdot\log(r) + \tilde{E}
\end{equation}
for big $r$ outside a set of finite linear measure.

Continuing with the next case, we assume that $\theta=\theta_2$. Then we get that $\ncharah{\theta(x-y)}{r}$ equals
\begin{equation}
\label{EqAnti8}
\ncharah{\left( \frac{p'(z)}{y-x} \left( \frac{\partial}{\partial x} - \frac{\partial}{\partial y} \right) + \frac{\partial}{\partial z}  \right) (x-y)}{r} = \ncharah{\frac{2p'(z)}{y-x}}{r}
\end{equation}
Using property (\ref{Nevanlinnaproperty1}) and (\ref{Nevanlinnaproperty3}) in Lemma~\ref{lem:holNevanlinnaproperties}, yields that the right hand side of \eqref{EqAnti8} is less or equal to
\begin{equation}
\label{EqAnti9}
\ncharah{(2p'(z)}{r} + \ncharah{x}{r} + \ncharah{y}{r} + E = (3n-1) \cdot \log(r) + E
\end{equation}
where in the last equality we used Lemma~\ref{lem:coordinates} and Corollary~\ref{cor:polycomp} for big $r$ outside a set of finite linear measure. Expression~\eqref{EqAnti9} together with expression~\eqref{EqAnti6} yields, following the chain of inequalities from expression~\eqref{EqAnti1}, that
\begin{equation}\label{ThetaEq2}
\ncharah{\theta(f)}{r} \leq 4 \cdot \ncharah{f}{r} + (11n+3)\cdot \log(r) + F
\end{equation}
for big $r$ outside a set of finite linear measure. Hence, using the correct constants from expressions \eqref{ThetaEq1} and \eqref{ThetaEq2} the results is proved.
\end{proof}

\noindent We are now ready to prove Theorem~\ref{MainEstimate}.

\begin{proof}[Proof of Theorem~\ref{MainEstimate}]
By expression~\eqref{VectorfieldSplit} we have that
\begin{equation}\label{MainEq1}
\ncharah{\theta(f)}{r} \leq \ncharah{\theta(f_\text{inv})}{r} + \ncharah{\theta(f_\text{anti})}{r} + C.
\end{equation}
By Lemma~\ref{InvariantResult} and Lemma~\ref{AntiInvariantResult} we have that the right hand side of \eqref{MainEq1} is less or equal to
\begin{equation}
\label{MainEq2}
3\cdot \ncharah{f_\text{inv}}{r} + 4 \cdot \log(r)  + 4 \cdot \ncharah{f_\text{anti}}{r} + E(n) \cdot \log(r) + F
\end{equation}
for big $r$ outside a set of finite linear measure. As $f_\text{inv} = \frac{f + f \circ I}{2}$ we get
\begin{equation}
\label{MainEq3}
\ncharah{f_\text{inv}}{r} = \ncharah{\frac{f + f \circ I}{2}}{r}
 \leq \ncharah{\frac{f}{2}}{r} + \ncharah{\frac{f \circ I}{2}}{r} + C
\end{equation}
by property \ref{Nevanlinnaproperty1} in Lemma~\ref{lem:holNevanlinnaproperties}. Using Lemma~\ref{InvolutionNevanlinna} yields that the right hand side of inequality~\eqref{MainEq3} equals
\begin{equation}
\label{MainEq4}
2 \cdot \ncharah{f}{r} + G
\end{equation}
for some constant $G$, so expression~\eqref{MainEq1} together with expressions \eqref{MainEq2} and \eqref{MainEq3} yields
\begin{multline}
\label{MainEq5}
\ncharah{\theta(f)}{r} \\
 \leq 3(2 \cdot \ncharah{f}{r} + G) + 4 \log r  + 4 \cdot \ncharah{f_\text{anti}}{r} + E(n) \cdot \log(r) + F
\end{multline}
Using that $f_\text{anti} = \frac{f \circ I-f}{2}$ yields, by similar estimates as for $f_\text{inv}$, that
\begin{equation}
\label{MainEq6}
\ncharah{f_\text{anti}}{r} \leq 2 \cdot \ncharah{f}{r} + H
\end{equation}
Combining equations \eqref{MainEq5} and \eqref{MainEq6} gives
\begin{multline*}
\ncharah{\theta(f)}{r} \\
 \leq 3(2 \cdot \ncharah{f}{r} + G) + 4 \cdot \log(r)  + 4 (2\cdot \ncharah{f}{r} + H) + E(n) \cdot \log(r) + F \\
  = 14 \cdot \ncharah{f}{r} + K(n) + L
\end{multline*}
for big $r$ outside a set of finite linear measure, where $K(n)$ is an affine polynomial and $L$ is a constant.
\end{proof}

The above theorem will be used for estimating the volume change of an overshear. Recall from the work of Kaliman and Kutzschebauch in~\cite{KutzschebauchKaliman} that the Danielewski surfaces admit a unique up to a nonzero constant algebraic ``volume form'' $\omega$, in the sense that is in every point of maximal complex rank. In the open dense chart corresponding to local coordinates on $\CC^*_x \times \CC_z \to D_p$  defined by $(x,z) \mapsto (x, \frac{p(z)}{x}, z)$, yields that the volume form is given by $\omega = \frac{dx \wedge dz}{x}$ (thus fixing the constant, which is however not important for the calculation of volume change). The volume change of a self map
$F : D_p \to D_p$  of the Danielewski surface,
\[
F(x,y,z) = (u(x,y,z), v(x,y,z), w(x,y,z))
\]
is due to the local expression of $\omega$ given by
\begin{equation}
\label{Jacobi}
\left|\text{Jac}(F)\right| = \frac{x}{u(x,z)} \left( \frac{\partial u}{\partial x} \frac{\partial w}{\partial z} - \frac{\partial u}{\partial z} \frac{\partial w}{\partial x} \right).
\end{equation}
in coordinates $(x,z)$ in the above chart.
Thus we need to estimate the growth of the derivative $\frac{\partial f}{\partial x} $ and $\frac{\partial f}{\partial z} $ by the growth of $f$ itself. To do so we push forward the vector fields $\frac{\partial }{\partial x}$ and $\frac{\partial }{\partial z} $ from $\CC^*_x \times \CC_z$ to the Danielewski surface, express them as a linear combination (with meromorphic coefficients) of the
fields $\theta_1$ and $\theta_2$  from equations \eqref{LiftedParialDerivative1} and \eqref{LiftedParialDerivative2} and use the corresponding estimates.

The pushed forward fields are the meromorphic vector fields
\begin{equation}
\label{MeromorphicVF1}
\theta = \frac{\partial}{\partial x} - \frac{y}{x} \frac{\partial}{\partial y}
\end{equation}
and
\begin{equation}
\label{MeromorphicVF2}
\tilde{\theta} = \frac{p'(z)}{x}\frac{\partial}{\partial y} + \frac{\partial}{\partial z}.
\end{equation}

The corresponding linear combinations are
\[\theta = \frac{x-y}{x}\theta_1\]
and
\[\tilde{\theta} = \theta_2 + \frac{p'(z)}{x}\theta_1\]
where $\theta_1$ and $\theta_2$ are the vector fields from expressions \eqref{LiftedParialDerivative1} and \eqref{LiftedParialDerivative2}. Given a holomorphic function $f : D_p \to \CC$ we have
\begin{equation}
\label{MeromorphicEq1}
\ncharah{\theta(f)}{r} = \ncharah{\frac{x-y}{x}\theta_1(f)}{r} \leq \ncharah{1-\frac{y}{x}}{r} + \ncharah{\theta_1(f)}{r}
\end{equation}
by property \ref{Nevanlinnaproperty2} in Lemma~\ref{lem:holNevanlinnaproperties}. Using Theorem~\ref{MainEstimate}, properties \ref{Nevanlinnaproperty2} and \ref{Nevanlinnaproperty3} in Lemma~\ref{lem:holNevanlinnaproperties}, and using Lemma~\ref{lem:coordinates}, yields that the right hand side of \eqref{MeromorphicEq1} is less or equal to
\begin{multline*}
\ncharah{y}{r} + \ncharah{x}{r} + C + 14 \cdot \ncharah{f}{r} + K(n) \cdot \log(r) + L \\
 = 2n \cdot \log(r) + C + \ncharah{f}{r} + K(n) \cdot \log(r) + L
\end{multline*}
where $C$ is the constant from property \ref{Nevanlinnaproperty3} in Lemma~\ref{lem:holNevanlinnaproperties}, for big $r$ outside a set of finite linear measure. Putting $\tilde{L}=L+C$ and $\tilde{K}(n) = K(n) + 2n$ we get that
\begin{equation}
\label{ThetaEstimate}
\ncharah{\theta(f)}{r} \leq 14 \cdot \ncharah{f}{r}+ \tilde{K}(n) \cdot \log(r) + \tilde{L}.
\end{equation}
Similar calculations with $\tilde \theta$ instead of $\theta$ yields that
\begin{equation}
\label{ThetaTildeEstimate}
\ncharah{\tilde\theta(f)}{r} \leq 28 \cdot \ncharah{f}{r} + \tilde{K}(n) \cdot \log(r) + \tilde{L}.
\end{equation}

\noindent Thus we have proved

\begin{proposition}
\label{wichtig}
Given a holomorphic function $f : D_p \to \CC$, and a vector field $\theta$ which is the push forward of a partial derivative from the chart $\alpha : \CC^*_x \times \CC_z \to D_p$ given by $(x,z) \mapsto \left(x, \frac{p(z)}{x}, z\right)$, then we have the estimate
\[
\ncharah{\theta(f)}{r} \leq A \cdot \ncharah{f}{r} + K(n) \cdot \log(r)
\]
for big $r$ outside a set of finite linear measure, where $K(n), A$ and $L$ are constants.
In other words:  If we denote the function
$f \circ \alpha$ on the chart again by $f$, then
\[
\ncharah{\frac {\partial f}{\partial x}}{r} \leq A \cdot \ncharah{f}{r} + K(n) \cdot \log(r) + L
\]
and
\[
\ncharah{\frac {\partial f }{\partial z}}{r} \leq A \cdot \ncharah{f}{r} + K(n) \cdot \log(r) + L
\]
for big $r$ outside a set of finite linear measure.
\end{proposition}

We will also use another coordinate chart on the Danielewski surface. Namely, around a  point $(0, 0, z_0)$ with $ p(z_0) = 0$ we can use $x$ and $y$ as coordinates, since $p' (z_0) \ne 0$. We do not give an explicit formula, but the chart exists by the inverse function theorem. In such coordinates the volume form $\omega$ is given by  $\frac 1 {p'(z)} dx \wedge dy$ and to estimate the volume change in this chart we use the following proposition below which follows exactly as Proposition \ref{wichtig}.

\begin{proposition} \label{auchwichtig}
Given a holomorphic function $f : D_p \to \CC$, and identify  $f$ and its derivatives with the corresponding functions on the $(x,y)$-chart then the estimates
\[
\ncharah{\frac{\partial f}{\partial x}}{r} \leq C \cdot \ncharah{f}{r} + D(n)\cdot\log(r) + E
\]
and
\[
\ncharah{\frac{\partial f}{\partial y}}{r} \leq C \cdot \ncharah{f}{r} + D(n)\cdot\log(r) + E
\]
for big $r$ outside a set of finite linear measure and for some constants $C, D(n)$ and $E$.
\end{proposition}

\section{Application to the Overshear Group}

In~\cite{AhernRudin}, Ahern and Rudin showed that the overshear group in $\CC^2$ is a free amalgamated product of the affine automorphisms and the elementary (or Jonquiere) automorphisms over their intersection. The analogous result for the polynomial automorphism group was shown by van der Kulk and Jung in~\cite{VanDerKulk} and~\cite{Jung}. In~\cite{MakarLimanov} the polynomial automorphism was determined in the sense that all its generators were given explicitly. Furthermore, Makar-Limanov showed, analogously to van der Kulk's and Jung's work, that $\text{Aut}_\text{pol}(D_p)$ has a structure of an amalgamated product. More precisely, he showed that
\begin{quote}
{\it Assume that $T_1$ is generated by shears $S_f$ and that $T_2$ is generated by shears $IS_fI$, where $f \in \CC[x]$ is a polynomial. Also consider $H$, the group generated by $H_\lambda(x,y,z)=(\lambda x, \lambda^{-1} y, z)$, and $I_2$ the group generated by the involution $I(x,y,z)=(y,x,z)$. Then}
$$\text{Aut}_\text{pol}(D_p) = T_1 * T_2 \rtimes (H \rtimes I_2) \cong T_1 * T_2 \rtimes (\CC^* \rtimes \ZZ_2).$$
\end{quote}
In the paper~\cite{MakarLimanov2}, Makar-Limanov studied $D_p^n$ and then he showed that
$$\text{Aut}_\text{pol}(D_p^n) = \CC[x] \rtimes \CC^*.$$
Both of these results were proved in a more general setting in~\cite{MakarLimanov} and \cite{MakarLimanov2}. Here we have restricted ourselves to the special case then $p$ has simple zeros only, since we want to consider only manifolds.

Ahern and Rudin used the Nevanlinna characteristic to show a similar structure theorem for the overshear group of $\CC^2$. Following the outline of their proof we will prove the following theorem.

\begin{theorem}
\label{StructureTheorem}
Let $D_p$ be a Danielewski surface and assume that $\deg(p) \geq 4$, then the overshear group, $\text{OS}(D_p)$, is a free amalgamated product of $O_1$ and $O_2$, where $O_1$ is generated by $O_{f,g}$ and $O_2$ generated by $IO_{f,g}I$.
\end{theorem}

Assume throughout this section that $n=\deg(p) \geq 4$.

\noindent The starting point of the proof of Theorem~\ref{StructureTheorem} uses the following lemma:

\begin{lemma}\label{SequenceLemma}
Every composition of overshear mappings is conjugate to the form
\begin{equation}\label{SequenceOfTriangular}
I \circ O_{f_1,g_1} \circ I \circ O_{f_2,g_2} \circ I \circ \cdots \circ O_{f_n,g_n}
\end{equation}
or to $I$.
\end{lemma}
\begin{proof}
We can assume that we have one of the following composition of overshear mappings:
\begin{equation}\label{SequenceOfTriangular1}
O_{f_1,g_1} \circ I \circ O_{f_2,g_2} \circ I \circ \cdots \circ I \circ O_{f_{n-1},g_{n-1}} \circ I \circ O_{f_n,g_n}\;,
\end{equation}
\begin{equation}\label{SequenceOfTriangular2}
O_{f_1,g_1} \circ I \circ O_{f_2,g_2} \circ \cdots \circ I \circ O_{f_n,g_n} \circ I\;,
\end{equation}
\begin{equation}\label{SequenceOfTriangular3}
I \circ O_{f_1,g_1} \circ I \circ O_{f_2,g_2} \circ \cdots \circ I \circ O_{f_{n-1},g_{n-1}} \circ I \circ O_{f_n,g_n}\;,
\end{equation}
\begin{equation}\label{SequenceOfTriangular4}
I \circ O_{f_1,g_1} \circ I \circ O_{f_2,g_2} \circ \cdots \circ I \circ O_{f_n,g_n} \circ I.
\end{equation}
We handle each of these compositions separately. We start with
expression~\eqref{SequenceOfTriangular1}. If we first conjugate with
$O_{f_1,g_1}$ and then with $I$ we obtain
$$O_{f_2,g_2} \circ I \circ \cdots \circ I \circ O_{f_{n-1},g_{n-1}} \circ I \circ O_{f_n\cdot f_1,g_1\cdot f_n+g_n} \circ I.$$
This is the desired form, but if $f_n \cdot f_1 \equiv 1$ and $g_1 \cdot f_n+g_n \equiv 0$,
then $O_{f_n\cdot f_1,g_1\cdot f_n+g_n} = \mathrm{id}$, and we are back where we started, {\it i.e.}
$$O_{f_2,g_2} \circ I \circ \cdots \circ I \circ O_{f_{n-1},g_{n-1}}.$$
The worst case scenario is that after a finite number of
conjugations with $O_{f_k,g_k}$ and $I$ we end up with
$$O_{f_j,g_j} \circ I \circ O_{f_{j+1},g_{j+1}}\;,$$
for some $j$. If we conjugate one last time we get
$$O_{f_{j+1} \cdot f_j, g_j\cdot f_{j+1} + g_{j+1}} \circ I.$$
This gives the conclusion of the lemma even if $f_{j+1} \cdot f_j \equiv 1$ and $g_j\cdot f_{j+1} + g_{j+1} \equiv 0$.

Expression~\eqref{SequenceOfTriangular2} is already of the desired
form, so let us consider expression~\eqref{SequenceOfTriangular3}
instead. In this equation it is suffices to conjugate with $I$.
Finally, consider expression~\eqref{SequenceOfTriangular4}. If we
start to conjugate with $I$, then we get a sequence as in
expression~\eqref{SequenceOfTriangular1}, and this case was handled
above.
\end{proof}

Look at an overshear mapping composed with an involution mapping on the Danielewski surface $D_p$
\begin{equation}\label{StructureFunctionEq1}
(u_1,v_1,w_1) = \left( \frac{p(zf_1(x)+xg_1(x))}{x}, x, zf_1(x)+xg_1(x) \right)
\end{equation}
and at a finite composition of such mappings
\begin{multline}\label{StructureFunctionEq2}
(u_{k+1},v_{k+1},w_{k+1}) \\ = \left( \frac{p(w_kf_{k+1}(u_k)+u_kg_{k+1}(u_k))}{u_k}, u_k,  w_kf_{k+1}(u_k)+u_kg_{k+1}(u_k) \right)\;,
\end{multline}
where $k \geq 1$, $f_j \colon \CC \to \CC^*$ is  holomorphic and $f(0)=1$ (thus either constantly $1$ or transcendental), and $g_j \colon \CC \to \CC$ is holomorphic.

Consider the mapping in expression~\eqref{StructureFunctionEq1}. Since $u_1 \cdot v_1 = p(w_1)$ on $D_p$, we have that $u_1 = \frac{p(w_1)}{v_1} = \frac{p(w_1)}{x}$. By expression~\eqref{Jacobi} we get that the determinant of the Jacobian of $(u_1,v_1,w_1)$ equals
\begin{equation}\label{CalculateJacobi}
\frac{x^2}{p(w_1)} \left( \frac{\partial u_1}{\partial x} \frac{\partial w_1}{\partial z} - \frac{\partial u_1}{\partial z} \frac{\partial w_1}{\partial x} \right).
\end{equation}
Since $u_1 = \frac{p(w_1)}{x}$, then
\begin{equation}\label{ddx}
\frac{\partial u_1}{\partial x} = \frac{p'(w_1)\frac{\partial w_1}{\partial x} x - p(w_1)}{x^2}
\end{equation}
and
\begin{equation}\label{ddz}
\frac{\partial u_1}{\partial z} = \frac{p'(w_1)\frac{\partial w_1}{\partial z}}{x}.
\end{equation}
Inserting expressions \eqref{ddx} and \eqref{ddz} in \eqref{CalculateJacobi} yields that
\begin{equation}
\label{CalculateJacobi2}
\frac{x^2}{p(w_1)} \left( \frac{p'(w_1)\frac{\partial w_1}{\partial x} \frac{\partial w_1}{\partial z} x - p(w_1)\frac{\partial w_1}{\partial z}}{x^2} - \frac{p'(w_1)\frac{\partial w_1}{\partial z}\frac{\partial w_1}{\partial x}}{x}  \right)
 = -\frac{\partial w_1}{\partial z}.
\end{equation}
Thus the chain rule implies that
\begin{lemma}
\label{Jacobirelation}
The following relation holds:
\[
\frac{|\text{Jac}(u_{k+1}, v_{k+1}, w_{k+1})|}{|\text{Jac}(u_k, v_k, w_k)|} = - f_{k+1}(u_k)
\]
\end{lemma}

An important step in the proof of Theorem~\ref{StructureTheorem} are the following lemmas.
The proofs of our two key lemmas uses estimates of the determinant of the Jacobian.

\begin{lemma}\label{JacobiEstimateLemma}
The functions from expressions \eqref{StructureFunctionEq1} and \eqref{StructureFunctionEq2} fulfil
\[
\ncharah{|\text{Jac}(u_k,v_k,w_k)|}{r} \leq A \cdot \ncharah{u_k}{r} + B \cdot \ncharah{v_k}{r} + C \cdot \log(r) + D
,\]
for $k \geq 1$, for certain constants $A, B, C, D$, for big $r$ outside a set of finite linear measure.
\end{lemma}
\begin{proof}
As
\[
|\text{Jac}(u_k,v_k,w_k)| = \frac{p'(z)}{p'(w_k)} \left( \frac{\partial u_k}{\partial x} \frac{\partial v_k}{\partial y} - \frac{\partial u_k}{\partial y} \frac{\partial v_k}{\partial x} \right)
\]
we get the following estimate by properties \ref{Nevanlinnaproperty1} and \ref{Nevanlinnaproperty2} in Lemma~\ref{lem:holNevanlinnaproperties}
\begin{multline}
\label{JacobiEq1}
\ncharah{|\text{Jac}(u_k,v_k,w_k)|}{r} \leq \ncharah{\frac{p'(z)}{p'(w_k)}}{r} + \ncharah{\frac{\partial u_k}{\partial x}}{r} \\
 + \ncharah{\frac{\partial v_k}{\partial y}}{r} + \ncharah{\frac{\partial u_k}{\partial y}}{r} + \ncharah{\frac{\partial v_k}{\partial x}}{r} + \log 2
\end{multline}
By property \ref{Nevanlinnaproperty3} in Lemma~\ref{lem:holNevanlinnaproperties} and by Proposition~\ref{auchwichtig} we get that the right hand side of \eqref{JacobiEq1} is less or equal to
\begin{multline}
\label{JacobiEq2}
\ncharah{p'(z)}{r} + \ncharah{p'(w_k)}{r} \\
 + 2 C\cdot \ncharah{u_k}{r} + 2C \cdot \ncharah{v_k}{r} + 4D \cdot \log(r) + \tilde{E}
\end{multline}
where $\tilde{E}$ is the constant $4 E + \log 2$, for big $r$ outside a set of finite linear measure. Using Corollary~\ref{cor:polycomp} yields that the right hand side of \eqref{JacobiEq2} equals
\begin{multline}
\label{JacobiEq3}
(n-1) \cdot \ncharah{z}{r} + (n-1) \cdot \ncharah{w_k}{r} \\
 + 2C \cdot \ncharah{u_k}{r} + 2C \cdot \ncharah{v_k}{r} + 4D \cdot \log(r) + \tilde{E}_2
\end{multline}
where $\tilde{E}_2$ is a constant. By Lemma~\ref{lem:coordinates} we get that expression~\eqref{JacobiEq3} tends to
\begin{multline}
\label{JacobiEq4}
(n-1)2\log(r) + (n-1)\cdot\ncharah{w_k}{r} \\
 + 2C \cdot \ncharah{u_k}{r} + 2C \cdot \ncharah{v_k}{r} + 4D \cdot \log(r) + \tilde{E}_2
\end{multline}
for big $r$ outside a set of finite linear measure. Using that $u_k v_k = p(w_k)$ and using property \ref{Nevanlinnaproperty2} in Lemma~\ref{lem:holNevanlinnaproperties} and Corollary~\ref{cor:polycomp} we get that
\begin{equation}
\label{JacobiEq5}
\ncharah{u_k v_k}{r} = \ncharah{p(w_k)}{r} = n \cdot \ncharah{w_k}{r} + F
\end{equation}
which implies
\begin{equation}
\label{JacobiEq6}
\ncharah{w_k}{r} \leq \frac{\ncharah{u_k}{r} + \ncharah{v_k}{r} - F}{n}.
\end{equation}
Combining the chain of inequalities from expression~\eqref{JacobiEq1} to expression \eqref{JacobiEq4} with expression~\eqref{JacobiEq6} yields
\[
\ncharah{|\text{Jac}(u_k,v_k,w_k)|}{r} \leq A \cdot \ncharah{u_k}{r} + B \cdot \ncharah{v_k}{r} + C \cdot \log(r) + D
\]
for big $r$ outside a set of finite linear measure and properly chosen constants $A, B, C, D$.
\end{proof}

\begin{lemma}\label{lem:Step1}
The functions $u_1$ and $v_1$ in equation~\eqref{StructureFunctionEq1} fulfil the estimate
\[
\frac{\ncharah{u_1}{r}}{\ncharah{v_1}{r}} \geq 2
\]
for big $r$ outside a set of finite linear measure, provided $\deg p = n \geq 3$.
\end{lemma}
\begin{proof}
We will consider two cases.
\begin{enumerate}
\item[Case 1:] $f_1$ is identically 1. \hfill

We have $u_1 (x,z) = \frac{p(z + x g_1(x))} {x}$ and $v_1 (x, z) = x$. Therefore
\begin{equation}\label{Step1Eq1}
\frac{\ncharah{\frac{p(z + x g_1(x))} {x}}{r}}{\ncharah{x}{r}} \geq \frac{ \ncharah{p(z+xg_1(x))}{r} - \ncharah{x}{r}}{\ncharah{x}{r}}
\end{equation}
by property \ref{Nevanlinnaproperty2} in Lemma~\ref{lem:holNevanlinnaproperties}. If $g$ is transcendental then Corollary~\ref{cor:polycomp}, property \ref{Nevanlinnaproperty1} in Lemma~\ref{lem:holNevanlinnaproperties} together with Proposition \ref{prop:transcendental} yields that
\begin{multline}
\label{eq:step1below}
\lim_{r \to \infty} \frac{\ncharah{u_1}{r}}{\ncharah{v_1}{r}}
 \geq \lim_{r \to \infty} \frac{n \cdot \ncharah{x g_1(x)}{r} - n \cdot \ncharah{z}{r} - \ncharah{x}{r} - nD}{\ncharah{x}{r}} \\
 = \infty - \frac{2n}{n} - 1 - 0 = \infty.
\end{multline}
Assume now that $g_1$ is a polynomial. Then $u_1$ is a polynomial in $x, y$ and $z$ with highest order term $x^{n-1} \cdot \left(g_1(x)\right)^n$. From Lemma \ref{lem:coordinates} and Corollary $\ref{cor:polycomp}$ it follows that
\begin{equation*}
\lim_{r \to \infty} \frac{\ncharah{u_1}{r}}{\ncharah{v_1}{r}}
 \geq (n-1) + n \cdot \deg(g_1) \geq n - 1 \geq 2
\end{equation*}
and the first case is finished.

\item[Case 2:] $f_1$ is transcendental. \hfill

Combining expressions \eqref{StructureFunctionEq1} and \eqref{CalculateJacobi2} we see that
\[
|\text{Jac}(u_1,v_1,w_1)| = -f(x)
.\]
Therefore Lemma~\ref{JacobiEstimateLemma} yields
\begin{equation}
\label{u1Eq2}
\ncharah{-f(x)}{r} \leq A \cdot \ncharah{u_1}{r} + B \cdot \ncharah{v_1}{r} + C \cdot \log(r) + D,
\end{equation}
for constants $A,B,C$ and $D$. Dividing with $\ncharah{v_1}{r}$ and $A$ in expression~\eqref{u1Eq2} yields
\begin{multline*}
\frac{\ncharah{u_1}{r}}{\ncharah{v_1}{r}}
\geq \frac{1}{A}\frac{\ncharah{-f(x)}{r}}{\ncharah{x}{r}} - \frac{B}{A} - \frac{C\log r +D}{A \cdot \ncharah{x}{r}}
 = \infty - \frac{B}{A} - 0 = \infty
\end{multline*}
for big $r$ outside a set of finite linear measure, by Proposition \ref{prop:transcendental}. \qedhere
\end{enumerate}
\end{proof}

\begin{lemma}
\label{lem:Stepk}
Let $u, v, w : D_p \to \CC$ be holomorphic functions defined on the Danielewski surface with $\deg p = n$, satisfying $u v = p(w)$.
Let $f : \CC \to \CC^*$ and $g : \CC \to \CC$ be holomorphic, with $f(0) = 1$. Assume that
\[
\frac{\ncharah{u}{r}}{\ncharah{v}{r}} > 1 + \delta \; \mbox{ for some } \delta \geq 0
\]
\begin{center}and\end{center}
\[
n \geq 5 \mbox{ or } \left(n = 4 \mbox{ and } \delta > 0\right)
\]
for big $r$ outside a set of finite linear measure. Then the functions $U$, $V$ defined by
\begin{equation}
\label{StepkEq}
(U,V,W) = \left( \frac{p(w f(u) + u g(u))}{u}, u, w f(u) + u g(u) \right)
\end{equation}
fulfil
\[
\frac{\ncharah{U}{r}}{\ncharah{V}{r}} > 1 + \varepsilon
\]
for big $r$ outside a set of finite linear measure too,
where
\[
\varepsilon = \frac{1}{2} \frac{\delta}{1 + \delta}.
\]
In fact, $\varepsilon$ can be chosen arbitrary large if $f$ or $g$ is transcendental, even if $\delta = 0$.
\end{lemma}
\begin{proof}
We start with an observation: The relation $u \cdot v = p(w)$ implies, together with Lemma \ref{lem:coordinates} and Proposition \ref{prop:Mohonko} that
\[
\ncharah{w}{r} \leq \frac{\ncharah{u}{r} + \ncharah{v}{r} - C}{n}, \quad C \in \RR
,\]
which results in
\begin{equation}
\label{FirstObs}
\frac{\ncharah{w}{r}}{\ncharah{u}{r}}
\leq \left(\frac{1}{n} \frac{\ncharah{v}{r}}{\ncharah{u}{r}} + \frac{1}{n}  - \frac{C}{n \cdot \ncharah{u}{r}} \right)
 < \frac{1}{n} \cdot \left(2 - \frac{\delta}{1 + \delta} \right),
\end{equation}
for big $r$ outside a set of finite linear measure by the hypothesis in the lemma. We will now consider two different cases as we did in Lemma~\ref{lem:Step1}.
\begin{enumerate}
\item[Case 1:]  $f$ is identically 1. \hfill
\begin{eqnarray*}
\frac{\ncharah{U}{r}}{\ncharah{V}{r}} = \frac{\ncharah{\frac{p(w + u \cdot g(u))}{u}}{r}}{\ncharah{u}{r}} & \geq & \frac{\ncharah{p(w + u \cdot g(u))}{r} - \ncharah{u}{r}}{\ncharah{u}{r}} \\
& \geq & \frac{n \cdot \ncharah{u \cdot g(u)}{r} - n \cdot \ncharah{w}{r}  - \ncharah{u}{r} + C}{\ncharah{u}{r}} \\
& \geq & n \cdot \frac{\ncharah{u \cdot g(u)}{r}}{\ncharah{u}{r}} - 3 + \frac{\delta}{1 + \delta} + O\left(\frac{1}{\log(r)}\right)
\end{eqnarray*}
\begin{enumerate}
\item If $g$ is transcendental, then the first term in the sum tends to infinity, $\varepsilon$ can be chosen arbitrary large.
\item If $g$ is a polynomial (in worst case, $g$ is constant), then we need to have that $n - 3 + \frac{\delta}{1 + \delta} > 1 + \varepsilon$. This can be satisfied with $\varepsilon = \frac{1}{2} \frac{\delta}{1 + \delta}$ if $n \geq 5$ or $n = 4$ and $\delta > 0$.
\end{enumerate}

\item[Case 2:] $f$ is transcendental \hfill

Lemma~\ref{Jacobirelation}, Lemma~\ref{JacobiEstimateLemma} and properties \ref{Nevanlinnaproperty2} and \ref{Nevanlinnaproperty3} in Lemma~\ref{lem:holNevanlinnaproperties} yield
\begin{multline}\label{TranscendentalEq}
\ncharah{-f(u)}{r} = \ncharah{\frac{|\text{Jac}(U,V,W)|}{|\text{Jac}(u,v,w)|}}{r} \\
 \leq A \cdot \ncharah{U}{r} + B \cdot \ncharah{V}{r} + C \cdot \ncharah{u}{r} \\
 + D \cdot \ncharah{v}{r} + E \cdot \log(r) + F.
\end{multline}
Divide expression~\eqref{TranscendentalEq} by $A \cdot \ncharah{V}{r}$. Then we get
\begin{equation}
\label{TranscendentalEq2}
\frac{\ncharah{U}{r}}{\ncharah{V}{r}} \geq \frac{1}{A}\frac{\ncharah{-f(u)}{r}}{\ncharah{V}{r}} - \frac{B}{A}
 - \frac{C\cdot\ncharah{u}{r}}{A\cdot\ncharah{V}{r}} - \frac{E \cdot \log(r) + F}{A \cdot \ncharah{V}{r}}
\end{equation}
for big $r$ outside a set of finite linear measure. As $\ncharah{V}{r} = \ncharah{u}{r}$, we get, by the hypothesis and Proposition \ref{prop:transcendental}, that the right hand side of \eqref{TranscendentalEq2} is greater than
\[
\infty - \frac{B}{A} -\frac{C}{A} - 0 = \infty
.\]
In this case we can choose $\varepsilon$ arbitrary large.
Hence, the proof is finished. \qedhere
\end{enumerate}
\end{proof}

We can interpret this lemma in the following way: In case of degree $n \geq 5$, the property $\displaystyle \frac{\ncharah{u}{r}}{\ncharah{v}{r}} > 1$ is preserved under composition with a shear map. In case of degree $n = 4$ this is only true if the fraction happens to be bounded away from $1$.

\begin{proof}[Proof of Theorem~\ref{StructureTheorem}]
We will show that $\id$ is not a reduced product. This will show that we cannot get non-trivial kernel of the mapping
\[
\varphi : O_1 * O_2 \to \text{OS}(D_p)
\]
defined in the obvious way. The mapping $\varphi$ is surjective by construction. To get a contradiction, we assume that $\id$ is a reduced product. By Lemma~\ref{SequenceLemma} we may assume that
\begin{equation}\label{StructEq}
\id = I \circ O_{f_1,g_1} \circ I \circ O_{f_2,g_2} \circ I \circ \cdots \circ I \circ O_{f_n,g_n} =: G_1 \circ \cdots \circ G_n,
\end{equation}
where $G_i = I \circ O_{f_i,g_i}$, for each $i\geq 1$. For $k \leq n$, define functions $u_k,v_k$ and $w_k$ on $D_p$ by
\[
(u_k(x,y,z),v_k(x,y,z),w_k(x,y,z)) = G_1 \circ \cdots \circ G_k(x,y,z)
.\]

The definition of $u_k, v_k$ and $w_k$ are easily verified to be the following functions:
\[
(u_1,v_1,w_1) = \left( \frac{p(zf_1(x)+xg_1(x))}{x}, x, z f_1(x) + x g_1(x) \right)
\]
and
\begin{multline*}
(u_{k+1},v_{k+1},w_{k+1}) \\ = \left(\frac{p(w_k f_{k+1}(u_k)+u_k g_{k+1}(u_k))}{u_k}, u_k, w_k f_{k+1}(u_k)+u_k g_{k+1}(u_k)\right).
\end{multline*}
These functions are the same as in expressions \eqref{StructureFunctionEq1} and \eqref{StructureFunctionEq2}.
Now Lemma \ref{lem:Step1} implies $\frac{\ncharah{u_1}{r}}{\ncharah{v_1}{r}} > 1 + \delta_1, \, \delta_1 > 0,$ for big $r$ outside a set of finite linear measure. By induction using Lemma \ref{lem:Stepk} it follows that $\frac{\ncharah{u_k}{r}}{\ncharah{v_k}{r}} > 1 + \delta_k, \, \delta_k > 0,$ for all $k$, for big $r$ outside a set of finite linear measure.  As $\lim_{r\to \infty} \frac{\ncharah{x}{r}}{\ncharah{y}{r}} = 1$, by Lemma \ref{lem:coordinates}, we get a contradiction. Therefore $\varphi$ is injective, and by construction also surjective. Hence $\text{OS}(D_p) = O_1 * O_2$.
\end{proof}

In general it is known that the overshear group on $\CC^n$ is a proper subgroup of the holomorphic automorphism group by an argument using Baire Category theorem due to Anders\'en and Lempert \cite{AndersenLempert}. However, only for $\CC^2$ there is a known explicit example of a holomorphic automorphism not belonging to the overshear group, see \cite{Andersen}. For the Danielewski surface we are also able to give such a concrete example. 

\begin{corollary} \label{proper} Let $p\ge 4$. The automorphism $(x,y,z) \mapsto (x e^{z}, y e^{-z}, z)$ is not contained in the overshear group, thus $\text{OS}(D_p)$ is a proper subgroup of the holomorphic automorphism group of $D_p$. 
\end{corollary}

\begin{proof}
We apply the same proof as for Theorem \ref{StructureTheorem}. It is sufficient to show that $\ncharah{x e^z }{r} = \ncharah{y e^{-z}}{r}$ which by the previous discussion makes it impossible to write this map as a composition of shears. By symmetry in $x$ and $y$, clearly $\ncharah{x e^z }{r} = \ncharah{y e^z}{r}$.
\begin{eqnarray*}
 \ncharah{y e^{-z}}{r} = \ncharhol{1/(y e^{-z})}{r} & = & \int_{\pi^{-1}(rS^3)} \log^+\left| y e^{-z} \right|^2 \pi^* \d\eta \\
                            & \stackrel{=}{z \mapsto -z} & \int_{\pi^{-1}(rS^3)} \log^+\left| y e^{z} \right|^2 \pi^* \d\eta \\
                            & = & \ncharhol{1/(y e^{z})}{r} = \ncharah{y e^z}{r}
\end{eqnarray*}
The last equality holds since $rS^3$ is invariant under this coordinate transformation and $\pi(x, y, z) = (x + y, z)$.
\end{proof}

\begin{remark}\label{ShearRemark}
Let $S_1$ be the subgroup of the shear group generated by $S_f$ and let $S_2$ be generated by $I\circ S_f \circ I$. Then it follows from Theorem~\ref{StructureTheorem}, that $\text{S}(D_p)$ is a free product of $S_1$ and $S_2$ for $\deg(p) \geq 4$. For $\deg(p) = 1$ the Danielewski surface is just $\CC^2$. In this case the conclusion of Theorem~\ref{StructureTheorem} does not hold. Indeed, the identity, viewed as a $2 \times 2$ matrix, can be written as the following product
\begin{equation*}
\left[ \begin{array}{cc} 1 & -1\\ 0 & 1\end{array}\right] \left[ \begin{array}{cc} 1 & 0\\ 1 & 1\end{array}\right]\left[ \begin{array}{cc} 1 & -1\\ 0 & 1\end{array}\right] \left[ \begin{array}{cc} 1 & 0\\ -1 & 1\end{array}\right] \left[ \begin{array}{cc} 1 & 1\\ 0 & 1\end{array}\right] \left[ \begin{array}{cc} 1 & 0\\ -1 & 1\end{array}\right]
\end{equation*}
of shears. If $\deg(p)=3$ the proof of Theorem~\ref{StructureTheorem} does not work, but the authors do not know whether the theorem holds or not. However for $\deg(p)=2$ there is the following counterexample:
\end{remark}
\begin{example}[Counterexample for $n=2$] \hfill \\
Let $p(z) = z^2 - 1 = (z + 1)(z - 1)$ and consider the following map
\[
A: \CC^3 \to \CC^3, \; (x,y,z) \mapsto (u,v,w) := (-x + y + 2iz, x, ix + z)
\]
This map induces by restriction a map $A|D_p : D_p \to D_p$, as one easily checks:
\[u v - p(w) = (y - x + 2iz) \cdot x - (z + i x)^2 + 1 = 0, \quad \mbox{ using } x y - z^2 + 1 = 0\]
By looking at the eigenvalues of the complex linear map $A$, one sees that $A^6 = \mathrm{id}_{\CC^3}$. Therefore the structure theorem cannot hold in this case.
\end{example}

\end{document}